\documentclass[11pt]{article}
\usepackage{amsmath, amsthm, amstext}
\usepackage{amssymb, array, amsfonts}
\usepackage[utf8]{inputenc}
\usepackage[english]{babel}
\usepackage{enumerate}
\usepackage{graphicx}
\usepackage{mathrsfs}

\usepackage{cite}

\inputencoding{utf8}

\newtheorem{thm}{Theorem}

\title{On a certain representation of a solution to the characteristic problem for the ultrahyperbolic equation}
\author{Maxim N. Demchenko\footnote{St.~Petersburg Department of
V.\,A.~Steklov Institute of Mathematics of
the Russian Academy of Sciences, 
27 Fontanka, St.~Petersburg, Russia. E-mail: demchenko@pdmi.ras.ru}}
\date{}

\begin{document}
\maketitle
\begin{flushright}
{\em\large
\begin{tabular}{r}
To the memory of Vassiliy Mikhaylovich Babich
\end{tabular}
}
\end{flushright}

\begin{abstract}
We consider the characteristic problem for the ultrahyperbolic equation in the Eucli\-dean space.
The value of a solution is prescribed on the characteristic hyperplane.
A well-posed set-up of the problem is discussed.
We obtain a certain representation for a solution
suitable for analysis of its asymptotics at the infinity.

\smallskip

\end{abstract}

\section{Introduction}
We consider the equation
\begin{equation}
  (\partial^2_{ts} + \partial_{x_1}^2 +\ldots + \partial_{x_d}^2 - \partial_{y_1}^2 - \ldots - \partial_{y_n}^2) v = 0,
  \label{eqn-psi}
\end{equation}
where $v$ is a function of variables
\begin{equation*}
  (t, s, x_1, \ldots, x_d, y_1, \ldots, y_n) \in {\mathbb R}\times {\mathbb R}\times{\mathbb R}^d\times{\mathbb R}^n,
  \quad
  d,n \geqslant 1,
\end{equation*}
subject to the following condition
\begin{equation}
  v(0, \cdot) = v_0.
  \label{cauchy}
\end{equation}
Equation~(\ref{eqn-psi}) is of ultrahyperbolic type,
and condition~(\ref{cauchy}) prescribes the value of a solution on the characteristic hyperplane $\{t=0\}$.

Problem~(\ref{eqn-psi}), (\ref{cauchy}) was shown to be well-posed in a certain class of functions
in~\cite{Blag}.
Also the following conservation law was established there
\begin{equation*}
  \|v(t, \cdot)\|_{L_2} = \|v_0\|_{L_2}, \quad t\in{{\mathbb R}}.
\end{equation*}
These two features show that the problem in consideration is close in its properties
to classical evolution problems for hyperbolic equations.
In particular, it seems to be possible to build an analogue of nonstationary scattering theory for such a problem
considering $t$ as a time parameter.
In this context, the issue of asymptotic behavior of solutions as $t\to\infty$ would play a significant role.
In the present paper, we obtain a representation for a solution to problem~(\ref{eqn-psi}), (\ref{cauchy}),
suitable for deduction of such an asymptotics (Theorem~\ref{thm-repr}).

In paper~\cite{Blag},
a formula for a solution to problem~(\ref{eqn-psi}), (\ref{cauchy}) was obtained in the form of a convolution
of the data $v_0$ with a certain distribution in ${\mathbb R}\times{\mathbb R}^d\times{\mathbb R}^n$ being an analogue
of a fundamental solution.
The latter is described in terms of an analytic continuation of a certain distribution
with respect to a complex parameter,
which complicates the application of this formula to the asymptotic analysis of the solution.

Our interest to the characteristic problem is motivated by the fact that,
in contrast to the Cauchy problem for ultrahyperbolic equations,
it allows a well-posed set-up in standard (non-analytic) classes of functions.
Besides the paper~\cite{Blag} mentioned above,
this issue was also studied in~\cite{Blag-char},
where the well-posedness of the problem with data on the characteristic cone was established.
Note also that equations of this type allow a set-up, in which a solution is subject to
asymptotic conditions at the infinity~\cite{DMN, DMN23, DMN-DD24, DMN24}.
In the case of hyperbolic equations, such kind of problems
were studied in~\cite{Blag-scatter, Moses, Kis, Plachenov}.

\section{The set-up and well-posedness of the problem} 
Introduce notation
\begin{gather*}
  \overline x = (x_1, \ldots, x_d) \in {\mathbb R}^d,
  \quad
  \overline y = y_1, \ldots, y_n \in {\mathbb R}^n,
  \quad
  N = d+n.
\end{gather*}

We will assume that a solution $v$ to problem~(\ref{eqn-psi}), (\ref{cauchy})
is a continuous function of variable $t\in{\mathbb R}$,
whose values are function of variables $(s, \overline x, \overline y)$ belonging to $L_2({\mathbb R}^{N+1})$.
This can be recorded as follows
\begin{equation}
  v \in C\left({\mathbb R}; L_2({\mathbb R}^{N+1})\right).
  \label{class}
\end{equation}
Such a function can be considered as a Lebesgue measurable scalar function of variables
$(t, s, \overline x, \overline y)$ in ${\mathbb R}^{N+2}$.
This follows from the fact that, for any $t\in{\mathbb R}$,
convolutions $\varepsilon^{-N-1} \chi_\varepsilon * v(t, \cdot)$
($\chi_\varepsilon$ is the characteristic function of the ball in ${\mathbb R}^{N+1}$ of radius $\varepsilon$ centered at the origin)
with respect to variables $s, \overline x, \overline y$
tend to $v(t, \cdot)$ as $\varepsilon\to 0$ at every Lebesgue point of this function from $L_2({\mathbb R}^{N+1})$.
Since these convolutions are continuous in ${\mathbb R}^{N+2}$,
the function $v$ is measurable being equal almost everywhere in ${\mathbb R}^{N+2}$ to the limit of a sequence of continuous functions.

Assumption~(\ref{class}) implies also that $v\in L_{2,{\rm loc}}({\mathbb R}^{N+2})$,
which means that this function is a regular distribution in ${\mathbb R}^{N+2}$.
Thus we may demand equation~(\ref{eqn-psi}) to be fulfilled in the sense of distributions in ${\mathbb R}^{N+2}$.
Condition~(\ref{cauchy}) with data $v_0\in L_2({\mathbb R}^{N+1})$ has the obvious meaning for functions from class~(\ref{class}).
In this section, we show that problem~(\ref{eqn-psi}), (\ref{cauchy}) in such a set-up has a unique solution.

\begin{thm}
  There exists a unique solution $v(t,s,\overline x, \overline y)$ from class~(\ref{class}) satisfying~(\ref{eqn-psi}), (\ref{cauchy})
  for a given function $v_0\in L_2({\mathbb R}^{N+1})$.
\end{thm}
\begin{proof}
  Define the Fourier transform of a function $f(s, \overline x, \overline y)$ as follows
\begin{equation*}
  (F f)(\lambda, \overline\xi, \overline\eta) = 2 \int_{{\mathbb R}^{N+1}} e^{-i(s\lambda + \overline x\,\overline\xi - \overline y\,\overline\eta)} f(s, \overline x, \overline y)\, ds d\overline x d\overline y.
\end{equation*}

We begin with the proof of uniqueness of a solution.
Let a function $v(t, s, \overline x, \overline y)$ be a solution to the problem with $v_0 = 0$.
Denote by $\tilde v(t, \lambda, \overline\xi, \overline\eta)$ its Fourier transform $F$
with respect to variables $s, \overline x, \overline y$.
Next, let
$\psi(t, s, \overline x, \overline y)$ be an arbitrary test function from $C_0^\infty({\mathbb R}^{N+2})$,
and $\Psi(t, \lambda, \overline\xi, \overline\eta)$ -- its inverse Fourier transform $F^{-1}$ with respect to variables $s, \overline x, \overline y$.
We have (here and further angle brackets denote a pairing of a distribution and a test function)
\begin{multline*}
  0 = \langle{}v, (\partial^2_{ts} + \Delta_{\overline x} - \Delta_{\overline y})\psi\rangle
  = \int_{\mathbb R} dt \int_{{\mathbb R}^{N+1}} (v (\partial^2_{ts} + \Delta_{\overline x} - \Delta_{\overline y})\psi)(t, s, \overline x, \overline y) \, ds d\overline x d\overline y
  \\= \int_{\mathbb R} dt \int_{{\mathbb R}^{N+1}} (\tilde v (-i\lambda \partial_t - \overline\xi^2 + \overline\eta^2)\Psi)(t, \lambda, \overline\xi, \overline\eta) \, d\lambda d\overline\xi d\overline\eta.
\end{multline*}
By Plancherel's theorem, the function $\tilde v$, as well as $v$, satisfies condition of the form~(\ref{class}),
hence it is also a Lebesgue measurable function in ${\mathbb R}^{N+2}$.
Besides, it follows from property~(\ref{class}) of $\tilde v$ that
the the integrand in the resulting expression is integrable in ${\mathbb R}^{N+2}$.
By applying Foubini's theorem, we obtain
\begin{equation*}
  \int_{{\mathbb R}^{N+1}} d\lambda d\overline\xi d\overline\eta \int_{\mathbb R} (\tilde v (-i\lambda \partial_t - \overline\xi^2 + \overline\eta^2)\Psi)(t, \lambda, \overline\xi, \overline\eta) \, dt = 0.
\end{equation*}
Now choose a test function of the form
\begin{equation*}
  \psi\left(t, s, \overline x, \overline y\right) = \psi_1(t) \psi_2\left(s, \overline x, \overline y\right),
\end{equation*}
where $\psi_1 \in C_0^\infty({\mathbb R})$, $\psi_2 \in C_0^\infty({\mathbb R}^{N+1})$.
Then the integral in the last identity takes the form
\begin{equation*}
  \int_{{\mathbb R}^{N+1}} d\lambda d\overline\xi d\overline\eta\, (F^{-1}\psi_2)\left(\lambda, \overline\xi, \overline\eta\right)
  \int_{\mathbb R} \tilde v(t, \lambda, \overline\xi, \overline\eta) (-i\lambda \partial_t - \overline\xi^2 + \overline\eta^2) \psi_1(t) \, dt,
\end{equation*}
and the identity itself can be recorded as follows
\begin{equation*}
  \int_{{\mathbb R}^{N+1}} ((G\psi_2) w)\left(\lambda, \overline\xi, \overline\eta\right)
  \, d\lambda d\overline\xi d\overline\eta = 0,
  \quad
  \psi_2 \in C_0^\infty({\mathbb R}^{N+1}).
\end{equation*}
Here the function $w(\lambda, \overline\xi, \overline\eta)$ and the operator $G$ are defined by the following relations
\begin{gather*}
  G\psi_2 = P F^{-1}\psi_2,\\
  w(\lambda, \overline\xi, \overline\eta) = \frac{1}{P\left(\lambda, \overline\xi, \overline\eta\right)}
  \int_{\mathbb R} \tilde v(t, \lambda, \overline\xi, \overline\eta) (-i\lambda \partial_t - \overline\xi^2 + \overline\eta^2) \psi_1(t) \, dt,
  \\
  P(\lambda, \overline\xi, \overline\eta) = 1 + \lambda^2 + \overline\xi^2 + \overline\eta^2.
\end{gather*}
The function $w$ belongs to $L_2({\mathbb R}^{N+1})$ by Cauchy-Schwartz inequality and in view of the fact that the function $\psi_1$ is compactly supported.
Now we extract the equality $w=0$ from this identity.
To achieve this, it is sufficient to prove that any function from $\varphi\in\mathcal{S}({\mathbb R}^{N+1})$
can be approximated in the topology of $\mathcal{S}({\mathbb R}^{N+1})$ by functions of the form
$G\psi_2$, where $\psi_2\in C_0^\infty({\mathbb R}^{N+1})$.
Let a function $\chi\left(s, \overline x, \overline y\right)$ belong to $C_0^\infty({\mathbb R}^{N+1})$
and be equal to unity at the origin.
For $\varepsilon>0$ we set
\begin{equation*}
  \chi^\varepsilon\left(s, \overline x, \overline y\right) = \chi\left(\varepsilon s, \varepsilon \overline x, \varepsilon \overline y\right),
  \quad
  \psi_2^\varepsilon = \chi^\varepsilon F(\varphi/P) \in C_0^\infty({\mathbb R}^{N+1}).
\end{equation*}
Then the functions $\psi_2^\varepsilon$ tend to the function $F(\varphi/P)$ in $\mathcal{S}({\mathbb R}^{N+1})$ as $\varepsilon\to 0$,
and so, 
$G\psi_2^\varepsilon = P F^{-1} \psi_2^\varepsilon$ tend to $\varphi$.

It follows that, for a fixed function $\psi_1$ and for a.e. $(\lambda, \overline\xi, \overline\eta)$, we have
\begin{equation}
  \int_{\mathbb R} \tilde v(t, \lambda, \overline\xi, \overline\eta) (-i\lambda \partial_t - \overline\xi^2 + \overline\eta^2) \psi_1(t) \, dt = 0.
  \label{vt-identity}
\end{equation}
Now let $\{\psi_1^k\}_{k=1}^\infty \subset C_0^\infty(-T, T)$ be a countable dense subset of the space
\begin{equation*}
  C^1_T
  = \{ \varphi \in C^1[-T,T]\,|\, \varphi(\pm T) = \varphi'(\pm T) = 0\}
\end{equation*}
for some $T>0$.
For $(\lambda, \overline\xi, \overline\eta)$ from a set of full measure in ${\mathbb R}^{N+1}$,
equality~(\ref{vt-identity}) holds with $\psi_1 = \psi_1^k$ for any $k\geqslant 1$.
Therefore, it holds for any $\psi_1 \in C^1_T$.
Hence, for some $V(\lambda, \overline\xi, \overline\eta)$ and a.e. $t\in[-T,T]$, we have
\begin{equation}
  \tilde v(t, \lambda, \overline\xi, \overline\eta) = e^{i t (\overline\eta^2 - \overline\xi^2)/\lambda}\, V(\lambda, \overline\xi, \overline\eta).
  \label{Cauchy-solve}
\end{equation}

Condition $\tilde v(t,\cdot) \in L_2({\mathbb R}^{N+1})$ implies that $V\in L_2({\mathbb R}^{N+1})$.
Then the right hand side of~(\ref{Cauchy-solve}) determines an element from class~(\ref{class})
satisfying $\tilde v(0, \cdot) = V$.
This follows from Lebesgue's Dominated Convergence Theorem.
Taking into account that $\tilde v(0, \cdot) = 0$, this implies $V = 0$,
which means that $\tilde v = 0$ in $[-T,T]\times{\mathbb R}^{N+1}$.
Since $T$ is arbitrary, we conclude that $v = 0$ in ${\mathbb R}^{N+2}$.

Now turn to the existence of a solution to problem~(\ref{eqn-psi}), (\ref{cauchy}).
Our argument will be based on equality~(\ref{Cauchy-solve}).
More specifically, we will show that, for a given $v_0 \in L_2({\mathbb R}^{N+1})$, the inverse Fourier transform $F^{-1}$ of the function
\begin{equation}
  \tilde v(t, \lambda, \overline\xi, \overline\eta) = e^{i t (\overline\eta^2 - \overline\xi^2)/\lambda} \, \tilde v_0(\lambda, \overline\xi, \overline\eta),
  \quad \tilde v_0 = F v_0,
  \label{v-sol}
\end{equation}
with respect to variables $\lambda, \overline\xi, \overline\eta$ is a solution to the problem.
As was already pointed out, such a function belongs to class~(\ref{class})
and satisfies equality $\tilde v(0, \cdot) = \tilde v_0$, i.e. condition~(\ref{cauchy}).

Next, let $\psi(t, s, \overline x, \overline y)$ be a test function from $\mathcal{S}({\mathbb R}^{N+2})$,
$\Psi(t, \lambda, \overline\xi, \overline\eta)$ be its inverse Fourier transform $F^{-1}$ with respect to variables $s, \overline x, \overline y$.
We have
\begin{multline*}
  \langle{}v, (\partial^2_{ts} + \Delta_{\overline x} - \Delta_{\overline y})\psi\rangle
  = \int_{\mathbb R} dt \int_{{\mathbb R}^{N+1}} (v (\partial^2_{ts} + \Delta_{\overline x} - \Delta_{\overline y})\psi)(t, s, \overline x, \overline y) \, ds d\overline x d\overline y
  \\= \int_{\mathbb R} dt \int_{{\mathbb R}^{N+1}} (\tilde v (-i\lambda \partial_t - \overline\xi^2 + \overline\eta^2)\Psi)(t, \lambda, \overline\xi, \overline\eta) \, d\lambda d\overline\xi d\overline\eta
  \\= \int_{{\mathbb R}^{N+1}} d\lambda d\overline\xi d\overline\eta\, \tilde v_0(\lambda, \overline\xi, \overline\eta) \int_{\mathbb R} e^{i t (\overline\eta^2 - \overline\xi^2)/\lambda} 
  (-i\lambda \partial_t - \overline\xi^2 + \overline\eta^2)\Psi(t, \lambda, \overline\xi, \overline\eta) \, dt = 0.
\end{multline*}
Therefore, equation~(\ref{eqn-psi}) is satisfied in the sense of distributions.
\end{proof}

\section{Full Fourier transform of the solution to problem~(\ref{eqn-psi}),~(\ref{cauchy})} 
Relation~(\ref{v-sol}) gives an expression for the Fourier transform of the solution to the problem with respect to variables $s, \overline x, \overline y$.
However, in the study of the asymptotics of the solution for large $t$,
a formula for the full Fourier transform of the solution is more preferable.
The latter can be derived from~(\ref{v-sol}).
A rigorous justification of this derivation is complicated by the fact that the resulting distribution in ${\mathbb R}^{N+2}$
is singular, and by the presence of a singularity in the exponent on the right hand side of~(\ref{v-sol}).

In order to find the Fourier transform of the solution, we turn to the function
\begin{equation}
  u(x,y) = v\left(t, s, \overline x, \overline y\right), 
  \label{uv}
\end{equation}
where
\begin{equation*}
  x = (x_0, \overline x) \in {\mathbb R}\times{\mathbb R}^d,
  \quad
  y = (y_0, \overline y) \in {\mathbb R}\times{\mathbb R}^n,
\end{equation*}
and variables $x_0, y_0$ are related to $t, s$ as follows
\begin{equation}
    x_0 = t+s, \quad y_0 = t-s.
    \label{xy-st}
\end{equation}
The function $u(x,y)$ satisfies equation
\begin{gather}
  (\Delta_x - \Delta_y) u = 0,
  \label{eq-u}
\end{gather}
where $\Delta_x = \partial_{x_0}^2 + \ldots + \partial_{x_d}^2$, $\Delta_y = \partial_{y_0}^2 + \ldots + \partial_{y_n}^2$.
We define its Fourier transform by the equality (understood in the sense of tempered distributions) 
\begin{equation*}
  \hat u(\xi, \eta) = \int_{{\mathbb R}^{N+2}} e^{-i (x\xi - y\eta)} u(x, y)\, dx dy.
\end{equation*}
Equation~(\ref{eq-u}) implies that ${\rm supp}\, \hat u \subset \overline{\mathcal{C}}$, where
\begin{equation*}
  \mathcal{C} = \left\{(\xi,\eta)\in ({\mathbb R}^{d+1}\times{\mathbb R}^{n+1})\setminus\{0\}\,|\, \xi^2 = \eta^2\right\}.
\end{equation*}
We will seek $\hat u(\xi,\eta)$ in the following form
\begin{equation*}
  \hat u(\xi, \eta) = \delta(\xi^2 - \eta^2)\, a(\xi/|\xi|, \eta/|\eta|, |\xi|),
\end{equation*}
where the amplitude $a(\zeta, \sigma, r)$ is defined on $\Sigma\times {\mathbb R}_+$, $\Sigma = S^d\times S^n$.
The right hand side of the last relation is understood as the tempered distribution
acting on a test function $\Phi(\xi,\eta)$ from the Schwartz class $\mathcal{S}({\mathbb R}^{N+2})$ by the rule
\begin{equation}
  \langle\hat u, \Phi\rangle = \frac{1}{2} \int_{\Sigma\times{\mathbb R}_+} r^{N-1} a(\zeta,\sigma,r) \Phi(r\zeta, r\sigma) \, d\zeta d\sigma dr.
  \label{u-Phi}
\end{equation}
Here $d\zeta$ and $d\sigma$ denote the surface measure on $S^d$ and $S^n$, respectively.
We choose the following function as an amplitude
\begin{equation}
  a(\xi, \eta) = 2\pi |\xi_0 + \eta_0|\, \tilde v_0(\xi_0 + \eta_0, \overline\xi, \overline\eta).
  \label{ampl}
\end{equation}
Note that definition~(\ref{u-Phi}) makes sense for such an amplitude, 
since the integrand is a continuous bounded function in $\Sigma\times{\mathbb R}_+$ with bounded support.

In the remainder of this section, we prove the following theorem.
\begin{thm}  \label{thm-repr}
  Let a function $v_0(s,\overline x, \overline y)$ belong to $\mathcal{S}({\mathbb R}^{N+1})$,
  and the distribution $u(x,y)$ be defined by equalities~(\ref{u-Phi}), (\ref{ampl}).
  Then the distribution $v(t,s,\overline x, \overline y)$ related to $u(x,y)$ by~(\ref{uv})
  belongs to class~(\ref{class}) and satisfies~(\ref{eqn-psi}), (\ref{cauchy}).
\end{thm}

Define the Fourier transform $F_t f$ of a function $f(t)$ as follows
\begin{equation*}
  (F_t f)(\rho) = \int_{\mathbb R} e^{-i t \rho} f(t) \, dt.
\end{equation*}
Denote by $\hat v(\rho, \lambda, \overline\xi, \overline\eta)$ the Fourier transform $F_t$ of the function
$\tilde v(t, \lambda, \overline\xi, \overline\eta)$ with respect to $t$.
Now show that relations~(\ref{uv}), (\ref{xy-st}) 
imply the following equality
\begin{equation}
  \hat u(\xi, \eta) = \hat v(\rho, \lambda, \overline\xi, \overline\eta),
  \label{F-F}
\end{equation}
where
\begin{equation}
  \xi = (\xi_0, \overline\xi) \in {\mathbb R}\times{\mathbb R}^d, \quad \eta = (\eta_0, \overline\eta) \in {\mathbb R}\times{\mathbb R}^n,
  \label{xieta}
\end{equation}
and the variables
$\xi_0$, $\eta_0$ and $\rho$, $\lambda$ are related as follows
\begin{equation}
  \lambda = \xi_0 + \eta_0, \quad \rho = \xi_0 - \eta_0.
  \label{xietamunu}
\end{equation}
We derive equality~(\ref{F-F}) for classical functions $u$, $v$, related by~(\ref{uv}),
which implies the same fact for distributions.
We have
\begin{equation*}
  x_0 \xi_0 - y_0\eta_0 = \frac{1}{2} \left(x_0 (\lambda+\rho) - y_0 (\lambda-\rho)\right) = \frac{1}{2} \left((x_0-y_0) \lambda + (x_0+y_0) \rho\right)
  = t\rho + s\lambda,
\end{equation*}
hence
\begin{equation*}
  x \xi - y \eta = t\rho + s\lambda + \overline x \overline\xi - \overline y\, \overline\eta.
\end{equation*}
It follows that
\begin{equation*}
  \hat u(\xi, \eta) 
  = 2 \int_{{\mathbb R}^{N+2}} e^{-i (t\rho + s\lambda + \overline x \overline\xi - \overline y\, \overline\eta)}
  v(t, s, \overline x, \overline y) \, dt ds d\overline x d\overline y
  = \hat v(\rho, \lambda, \overline\xi, \overline\eta).
\end{equation*}

Let $\psi(t, \lambda, \overline\xi, \overline\eta)$ be a test function,
and $\Psi(\rho, \lambda, \overline\xi, \overline\eta)$ be its inverse Fourier transform $F_t^{-1}$ with respect to $t$.
Set
\begin{equation*}
  \Phi(\xi, \eta) = \Psi(\rho, \lambda, \overline\xi, \overline\eta),
\end{equation*}
provided that the variables $\rho, \lambda$ and $\xi_0, \eta_0$ are related according to~(\ref{xietamunu}).
We have
\begin{equation*}
  \langle\tilde v, \psi\rangle = \langle\hat v, \Psi\rangle = 2 \langle\hat u, \Phi\rangle
  = \int_{\Sigma\times{\mathbb R}_+} r^{N-1} a(\zeta,\sigma,r) \Phi(r\zeta, r\sigma) \, d\zeta d\sigma dr.
\end{equation*}
The coefficient $2$ in the second equality is the Jacobian of the change of variables
$(\rho, \lambda, \overline\xi, \overline\eta)$ to $(\xi, \eta)$.
In the third equality, we used definition~(\ref{u-Phi}).

Now choose a function $\chi(r)$ from $C_0^\infty([0,\infty))$ being equal to unity for small $r$.
In view of fast decay of the function $\Phi$, the resulting expression in the preceding calculation equals
\begin{equation*}
  \lim_{\varepsilon\to 0} \int_{\Sigma\times{\mathbb R}_+} r^{N-1} \chi(\varepsilon r) a(\zeta,\sigma,r) \Phi(r\zeta, r\sigma) \, d\zeta d\sigma dr.
\end{equation*}
Using notations for $\zeta, \sigma$, analogous to~(\ref{xieta}), we rewrite the integral in this expression as follows
\begin{multline}
  \int_{\Sigma\times{\mathbb R}_+} r^{N-1} \chi(\varepsilon r) a(\zeta,\sigma,r) \Psi(r(\zeta_0-\sigma_0), r(\zeta_0+\sigma_0), r\overline\zeta, r\overline\sigma) d\zeta d\sigma dr
  \\= \frac{1}{2\pi}\int_{\Sigma\times{\mathbb R}_+} d\zeta d\sigma dr\, r^{N-1} \chi(\varepsilon r) a(\zeta,\sigma,r) \int_{\mathbb R} e^{i t r (\zeta_0-\sigma_0)} \psi(t, r(\zeta_0 + \sigma_0), r\overline\zeta, r\overline\sigma) dt
  \\= \frac{1}{2\pi}\int_{\mathbb R} dt \int_{\Sigma\times{\mathbb R}_+} r^{N-1} \chi(\varepsilon r) a(\zeta,\sigma,r) e^{i t r (\zeta_0-\sigma_0)} \psi(t, r(\zeta_0 + \sigma_0), r\overline\zeta, r\overline\sigma)\, d\zeta d\sigma dr
  \label{int-a-phi}
\end{multline}
(the existence of the resulting integral is provided by the factor $\chi$).
Next we convert the inner integral with respect to $\zeta$, $\sigma$, $r$ to the integral with respect to variables
\begin{equation*}
  \lambda = r(\zeta_0+\sigma_0), \quad \overline\xi = r\overline\zeta, \quad \overline\eta = r\overline\sigma,
\end{equation*}
with the use of the formula (derived in sec.~\ref{sec-int-int})
\begin{multline}
  \int_{\Sigma\times{\mathbb R}_+} W(r\zeta, r\sigma) r^{N-1} d\zeta d\sigma dr
  =\\= \int_{{\mathbb R}^d\times{\mathbb R}^n\times{\mathbb R}}
  W\left(\frac{\overline\eta^2 - \overline\xi^2}{2\lambda} + \frac{\lambda}{2}, \overline\xi, \frac{\overline\xi^2 - \overline\eta^2}{2\lambda} + \frac{\lambda}{2}, \overline\eta\right)
  \frac{d\overline\xi d\overline\eta d\lambda}{|\lambda|},
  \label{int-int}
\end{multline}
in which $W(\xi, \eta)$ is a bounded continuous function on the surface $\mathcal{C}$ with bounded support.
In our case, this function equals
\begin{multline*}
  W(\xi, \eta)
  = \chi(\varepsilon |\xi|)\, a(\xi/|\xi|, \eta/|\eta|, |\xi|)\, e^{i t (\xi_0-\eta_0)} \psi(t, \xi_0 + \eta_0, \overline\xi, \overline\eta)
  \\=  2\pi \chi(\varepsilon |\xi|)\, |\xi_0 + \eta_0|\, \tilde v_0(\xi_0 + \eta_0, \overline\xi, \overline\eta) e^{i t (\xi_0-\eta_0)} \psi(t, \xi_0 + \eta_0, \overline\xi, \overline\eta)
\end{multline*}
(we applied formula~(\ref{ampl}))
and has a bounded support due the factor $\chi$.
Thus the resulting expression in~(\ref{int-a-phi}) equals
\begin{equation*}
  \int_{\mathbb R} dt \int_{{\mathbb R}^d\times{\mathbb R}^n\times{\mathbb R}}
  \chi(\varepsilon r)\, \tilde v_0(\lambda, \overline\xi, \overline\eta)\, e^{i t (\overline\eta^2 - \overline\xi^2)/\lambda} \psi(t, \lambda, \overline\xi, \overline\eta)
  \,d\overline\xi d\overline\eta d\lambda,
\end{equation*}
where $r = \left(\xi_0^2 + \overline\xi^2\right)^{1/2}$, and the function $\xi_0 = \xi_0(\overline\xi, \overline\eta, \lambda)$ is defined by formula~(\ref{xi-eta-0}).
We have $\chi(\varepsilon r) \to 1$ as $\varepsilon\to 0$, providing that $\overline\xi \ne 0$.
Since the function $\psi$ is compactly supported,
applying Dominated Convergence Theorem gives
\begin{equation*}
  \langle\tilde v, \psi\rangle =
  \int_{\mathbb R} dt \int_{{\mathbb R}^d\times{\mathbb R}^n\times{\mathbb R}}
  \tilde v_0(\lambda, \overline\xi, \overline\eta)\, e^{i t (\overline\eta^2 - \overline\xi^2)/\lambda} \psi(t, \lambda, \overline\xi, \overline\eta)
  \,d\overline\xi d\overline\eta d\lambda,
\end{equation*}
which implies equality~(\ref{v-sol}).

\section{Derivation of equality~(\ref{int-int})}   \label{sec-int-int}
We will use the following equality for a sufficiently regular function $f(\zeta)$ on the sphere~$S^m$
\begin{equation*}
  \int_{S^m} f(\zeta)\, d\zeta = \sum_\pm \int_{B^d_1} f\left(\pm\sqrt{1-\overline\zeta^2}, \overline\zeta\right) \frac{d\overline\zeta}{\sqrt{1-\overline\zeta^2}}
\end{equation*}
(henceforth $B^m_r$ is the open ball of radius $r$ in ${\mathbb R}^m$ centered at the origin).

Convert the left hand side of equality~(\ref{int-int}) as follows
\begin{multline}
  \int_0^\infty dr\, r^{N-1} \int_\Sigma W(r\zeta, r\sigma) \,d\zeta d\sigma
  \\= \sum_{\alpha,\beta = \pm 1} \int_0^\infty dr \, r^{N-1} \int_{B^d_1} \frac{d\overline\zeta}{\sqrt{1-\overline\zeta^2}}
  \int_{B^n_1} W\left(\alpha r \sqrt{1-\overline\zeta^2}, r\overline\zeta, \beta r \sqrt{1-\overline\sigma^2}, r\overline\sigma\right) \frac{d\overline\sigma}{\sqrt{1-\overline\sigma^2}}
  \\= \sum_{\alpha,\beta = \pm 1} \int_0^\infty dr\, r \int_{B^d_r} \frac{d\overline\xi}{\sqrt{r^2-\overline\xi^2}}
  \int_{B^n_r} W\left(\alpha \sqrt{r^2-\overline\xi^2}, \overline\xi, \beta \sqrt{r^2-\overline\eta^2}, \overline\eta\right) \frac{d\overline\eta}{\sqrt{r^2-\overline\eta^2}}
  \\= \sum_{\alpha,\beta = \pm 1} \int_{{\mathbb R}^d} d\overline\xi 
  \int_{{\mathbb R}^n} d\overline\eta \int_{\max(|\overline\xi|, |\overline\eta|)}^\infty 
  \frac{W(\xi_0, \overline\xi, \eta_0, \overline\eta) r dr}{|\xi_0 \eta_0|}.   
  \label{long-int}
\end{multline}
In the last expression, we used notations
\begin{equation*}
  \xi_0 = \alpha \sqrt{r^2-\overline\xi^2}, \quad \eta_0 = \beta \sqrt{r^2-\overline\eta^2}.
\end{equation*}
Next we make the following change of variable in the inner integral with respect to $r$
\begin{equation*}
  \lambda(r) = \xi_0 + \eta_0.
\end{equation*}
It suffices to consider the case when $\overline\xi^2 \ne \overline\eta^2$.
Under this assumption, the change of variable produces a one-to-one map, since the derivative
\begin{equation*}
  \partial_r\lambda = \frac{\alpha r}{\sqrt{r^2 - \overline\xi^2}} + \frac{\beta r}{\sqrt{r^2 - \overline\eta^2}}
\end{equation*}
does not vanish. 
Besides, we have $\lambda \ne 0$, and the following equality holds true 
\begin{equation*}
  \partial_\lambda r = (\partial_r\lambda)^{-1} = \frac{\alpha\beta\sqrt{(r^2-\overline\xi^2)(r^2-\overline\eta^2)}}{r\lambda}
  = \frac{\xi_0 \eta_0}{r\lambda}.
\end{equation*}

It follows from equalities
\begin{equation*}
  \xi_0^2 + \overline\xi^2 = \eta_0^2 + \overline\eta^2, \quad \eta_0 = \lambda - \xi_0
\end{equation*}
that for all $\alpha,\beta=\pm 1$ we have
\begin{equation}
  \xi_0 = \frac{\overline\eta^2 - \overline\xi^2}{2\lambda} + \frac{\lambda}{2},
  \quad
  \eta_0 = \frac{\overline\xi^2 - \overline\eta^2}{2\lambda} + \frac{\lambda}{2}.
  \label{xi-eta-0}
\end{equation}

Consider the case $\overline\xi^2 > \overline\eta^2$, in which
$r$ takes values from $(|\overline\xi|, \infty)$.
Next we describe the behavior of the function $\lambda(r)$ for various values of $\alpha,\beta$:
\begin{equation*}
  \begin{array}{rll}
    \alpha = \beta = 1: & \lambda(|\overline\xi|) = \Lambda, & \lambda(+\infty) = +\infty,
    \\
    \alpha = -\beta = 1: &   \lambda(|\overline\xi|) = -\Lambda, & \lambda(+\infty) = 0,
    \\
    \alpha = -\beta = -1: &   \lambda(|\overline\xi|) = \Lambda, & \lambda(+\infty) = 0,
    \\
    \alpha = \beta = -1: &   \lambda(|\overline\xi|) = -\Lambda, & \lambda(+\infty) = -\infty.
  \end{array}
\end{equation*}
Here $\Lambda = \sqrt{\overline\xi^2-\overline\eta^2}$.
In each of these cases, the inner integral in~(\ref{long-int}) equals the integral
\begin{equation*}
  \int W(\xi_0, \overline\xi, \eta_0, \overline\eta) \frac{d\lambda}{|\lambda|}
\end{equation*}
taken over the corresponding interval.
Since these intervals cover the real line without overlapping,
the inner integral in~(\ref{long-int}), after summing with respect to $\alpha,\beta=\pm 1$, in view of relations~(\ref{xi-eta-0}), equals
\begin{equation*}
  \int_{\mathbb R} W\left(\frac{\overline\eta^2 - \overline\xi^2}{2\lambda} + \frac{\lambda}{2}, \overline\xi, \frac{\overline\xi^2 - \overline\eta^2}{2\lambda} + \frac{\lambda}{2}, \overline\eta\right)
  \frac{d\lambda}{|\lambda|}.
\end{equation*}

In the case $\overline\xi^2 < \overline\eta^2$, analogous calculation yields the same expression for the integral with respect to $r$.

Now we may conclude calculation~(\ref{long-int}) by equality~(\ref{int-int}).
Absolute convergence of the integral on the right hand side of the latter follows from the derivation of this relation.
It can also be established directly.
Indeed, for large $\overline\xi$, $\overline\eta$, the integrand vanishes, since $W$ is a compactly supported function.
This is true also for large $\lambda$, since the sum of arguments $\xi_0$ and $\eta_0$ of the function $W$ equals $\lambda$.
Next, for small $\lambda$, the integrand is non-zero only if
\begin{equation*}
  |\overline\eta^2 - \overline\xi^2| \leqslant C |\lambda|
\end{equation*}
($C$ depends on the size of the support of the function $W$),
since otherwise, the arguments $\xi_0$, $\eta_0$ of the function $W$ under the integral sign would be sufficiently large.
The last inequality implies that
\begin{equation*}
  \left||\overline\eta| - |\overline\xi|\right| \leqslant \sqrt{C |\lambda|}.
\end{equation*}
Hence, for small $\lambda$, the measure of the set of points $(\overline\xi, \overline\eta)$, where the integrand is non-zero,
does not exceed $C \sqrt{|\lambda|}$, which implies the absolute convergence of the integral.

\end{document}